\documentclass{article}
\usepackage{amsmath}
\usepackage{graphicx} 
\usepackage{xcolor}
\usepackage[margin=.75in]{geometry}
\usepackage{amsmath}
\usepackage{amssymb}
\usepackage{mathtools}
\usepackage{parskip}
\usepackage{mathrsfs}
\usepackage{setspace}
\usepackage{float}
\usepackage[hidelinks]{hyperref}
\usepackage[noabbrev,capitalize,nameinlink]{cleveref}
\usepackage[backend=biber, style=numeric]{biblatex}

\usepackage[compact]{titlesec}
\usepackage{amsthm}
\usepackage{multicol}
\usepackage{graphicx}
\usepackage{fancyhdr}
\usepackage{esint}
\usepackage[b]{esvect}
\usepackage{letltxmacro}
\usepackage{enumitem}
\usepackage{tikz-cd}
\newtheorem{theorem}{Theorem}[section]
\newtheorem{theorem**}{Theorem}
\newtheorem{theorem***}{Theorem}
\newtheorem*{theorem*}{Theorem}
\newtheorem{corollary}{Corollary}[theorem]
\newtheorem{proposition}{Proposition}[theorem]
\newtheorem{lemma}[theorem]{Lemma}

\title{Strichartz Estimates for One Dimensional Stochastic Schrodinger Equation with Potential}
\author{Abhinav Goel\\{\small Mentor: Joshua Messing}\\{\small Project suggested by: Gigliola Staffilani}}
\date{June 2024}

\begin{document}

\maketitle

\begin{center}
\section{\centering Abstract}
This paper proves Strichartz estimates for the Schrodinger Equation with a potential term and white noise dispersion. We restrict our equation to dimension $1$. We consider dispersive estimates using work from Goldberg et al.\cite{Goldberg}., as well as Strichartz estimates using work from de Bouard et al. \cite{DeBouard}.

\end{center}

\section{\centering Introduction}

The Schrodinger Equation is a partial differential equation commonly used in Quantum Mechanics. It is used to model the positions and behaviors of particles as well as time evolutions of wave functions. It takes the form

\begin{equation} \label{eq1}
\begin{cases}
iu_t  = Hu \\
u_{t=0} = u_0 
\end{cases}
\end{equation}

where the Hamiltonian $H = -\Delta + V$ for some potential function $V$. 

We also consider a stochastic version of the equation,

\begin{equation} \label{eq2}
\begin{cases}
iu_t + \Delta u \circ \beta(t)  = 0 \\
u_{t=0} = u_0 
\end{cases}
\end{equation}

where $\beta(t)$ is a standard one-dimensional Brownian motion and $\circ$ denotes the Stratanovich product. This equation is commonly used in the field of fiber optics where light is being guided through optical (glass) fibers. This has applications ranging from internet and television to medical devices. Equation \ref{eq2} is used to describe signal propagation in optical fibers with dispersion management \cite{Agrawal,DeBouard}. Light can be composed of many different wavelengths, with each wavelength having different behaviors in the fiber. Additionally, distortion can be caused by elements in the fiber itself. Thus, dispersion management is used to limit these effects.

Combining the above two forms of equations, we consider the following equation:

\begin{equation} \label{eq3}
\begin{cases}
i \, \text{d}u = Hu \circ \, \text{d} \beta  \\
u_{t=0} = u_0 
\end{cases}
\end{equation}

Here, we replace the Laplacian from equation \ref{eq2} with the Hamiltonian, leading to an interaction between the potential and the Brownian motion. Studying this equation is motivated by the Euler-Heisenberg interaction, which dictates the bending of light when interacting with an electric charge \cite{Kim}. While the bending of light has been studied with respect to forces like gravity, a small electric charge peculiarly leads to a non-linear interaction. This creates a more complicated model, modeled similarly to equation $3$.  We will now investigate this equation providing a decay rate of $|t|^{\frac{-1}{2}}$ for the $L^{\infty}$ norm of the solution with $L^1$ initial solution.

We use work by Goldberg et al. \cite{Goldberg}, which examines equation \ref{eq1} when in dimensions $1$ and $3$ . By projecting on the spectral measure of the Hamiltonian, it finds a dispersive estimate of $|t|^{\frac{-1}{2}}$ for the solution operator. We use their approach to generate a dispersive estimate for equation \ref{eq3} below.

\begin{theorem**}
As shown below in section $3$, the solution operator to equation \ref{eq3} is $e^{-i\beta(t)H}$. Let $L_j^1(\mathbb{R}) \coloneq \{f: \int_{-\infty}^{\infty} |f(x)|(1+|x|)^j \, \text{d}x < \infty \}$. Assume $V \in L_1^1(\mathbb{R})$, and let $P_{ac}$ be the projection onto the absolutely continuous spectral subspace. If our system has no resonance at zero energy (as defined below in section $3$), then,

$$||e^{-i\beta(t)H}P_{ac}(H)||_{1 \rightarrow \infty} \lesssim |\beta(t)|^{\frac{-1}{2}}$$
\end{theorem**}

We also use work by de Bouard et al. \cite{DeBouard}, which creates Strichartz estimates for equation \ref{eq2}. By using Theorem $1$, along with the estimates from de Bouard, we can create Strichartz estimates for equation \ref{eq3}.

\begin{theorem**}
    Let $2 \leq r < \infty$ and $2 \leq p \leq \infty$ such that $\frac{2}{r} > \frac{1}{2} - \frac{1}{p}$ or $r=\infty$ and $p=2$. Let $S(\cdot,\sigma)$ be our solution operator for equation $3$ given initial condition at $t = \sigma$. Additionally, set $\rho$ such that $r' \leq \rho \leq r$, where $r'$ is the Holder conjugate of $r$. Then, there exists a constant $c_{\rho,r,p}$ such that for any $s \in \mathbb{R}$, $T \leq 0$, and $f \in L^{\rho}_{\mathcal{P}}(\Omega; L^{r'}(s,s+T;L_x^{p'}))$

    $$|\int_s^{\cdot} S(\cdot,\sigma)P_{ac}(H)f(\sigma) \, \text{d} \sigma|_{L^{\rho}(\Omega; L^{r}(s,s+T;L_x^{p}))} \leq c_{\rho,r,p}T^{\mu}|f|_{L^{\rho}(\Omega; L^{r'}(s,s+T;L_x^{p'}))} $$
with $\mu = \frac{2}{r} + \frac{1}{2p} - \frac{1}{4}$
\end{theorem**}

\begin{theorem**}
    Let $2 \leq r < \infty$ and $2 \leq p \leq \infty$ be such that $\frac{2}{r} > \frac{1}{2} - \frac{1}{p}$ or $r = \infty$ and $p = 2$. Then there exists a constant $c_{r,p} > 0$ such that for $s \in \mathbb{R}, T \geq 0$, $u_s \in L^r(\Omega;L_x^2)$,

    $$|S(\cdot,s)P_{ac}(H)u_s|_{L^r(\Omega;L^r(s,s+T;L^p_x))} \leq cT^{\frac{\mu}{2}}|u_s|_{L^r_{\omega}(L_x^2)}$$

    where $\mu = \frac{2}{r} - \frac{1}{2}(\frac{1}{2} - \frac{1}{p})$

    \text{Specifically},

    $$|e^{-i\beta(t)H}P_{ac}(H)u_0|_{L^r(\Omega;L^r(0,T;L^p_x))} \leq cT^{\frac{\mu}{2}}|u_0|_{L^r_{\omega}(L_x^2)}$$

    \end{theorem**}

\section{\centering Dispersive Estimates}
Consider the equation 

$$i\, \text{d}u = (-\Delta+V)u \circ \, \text{d}\beta$$

with $u \in \mathbb{R}$, $V \in L^1_1(\mathbb{R})$, $\beta(t)$ a standard one-dimensional Brownian motion,  and some initial condition $u_0$. $V \in L^1_1(\mathbb{R})$ means $\int_{-\infty}^{\infty} |V(x)|(1+|x|) \, \text{d}x < \infty$.

We can approach this equation by converting to its Ito form, 

$$i\, \text{d}u = \frac{-i}{2}(-\Delta+V)^2u \, \text{d}\text{t} + (-\Delta+V)u \, \text{d}\beta$$

In the integral form, this evaluates to

$$i(u(t)-u(0)) + \frac{i}{2} \int_0^t (-\Delta+V)^2u \, \text{d}\text{t} - \int_0^t (-\Delta+V)u \, \text{d}\beta  =0
 $$

Define the Hamiltonian, $H$, as $-\Delta+V$.

Let us consider the function 
$$g(t,x) = u_0(x) \cdot \int_{\sigma(H)} e^{-i\beta(t)\lambda} \, \text{d}H(\lambda)$$

where $\, \text{d}H$ is the spectral measure. 

We can apply Ito's lemma to see 

$$ig(t,x) = ig(0,x) + i\int_0^t \int (-i\lambda)e^{-i\beta_s\lambda}\, \text{d}H(\lambda) u_0(x)\, \text{d}\beta_s + \frac{i}{2} \int_{0}^{t} \int (-\lambda^2) e^{-i\beta_s\lambda} \, \text{d}H(\lambda)u_0(x)\, \text{d}S$$

$$ = ig(0,x) + \int_0^t Hg(s,x) \, \text{d}\beta_s - \frac{i}{2} \int_0^t H^2 g(s,x) \, \text{d}t$$

Thus, $i \, \text{d}g = Hg \, \text{d} \beta_s - \frac{i}{2}H^2g \, \text{d}t$ and so $g$ is a solution to equation \ref{eq3}. We see our solution operator is $e^{-i\beta(t)H}$ as defined by the functional calculus.

For $V \in L^1(\mathbb{R})$, $H$ is essentially self-adjoint on the domain

$$\{ f \in L^2(\mathbb{R}) | f,f' \text{are a.c and } -f'' + Vf \in L^2(\mathbb{R})\}$$ 

This means that $e^{-iH\beta(t)}$ is a unitary operator. Additionally, define $P_{ac}$ as the projection onto the absolutely continuous spectral subspace.

\begin{theorem***}
Assume our system has no resonance at zero energy. Then,

$$||e^{-i\beta(t)H}P_{ac}(H)||_{1 \rightarrow \infty} \lesssim |\beta(t)|^{\frac{-1}{2}}$$
\end{theorem***}

\begin{corollary}
For $1 \leq p < 2$
$$||e^{-i\beta(t)H}P_{ac}(H)||_{L^p_{\Omega}L_x^1 \rightarrow L^p_{\Omega}L_x^{\infty}} \lesssim |t|^{\frac{-1}{4}}$$

where the $L^p$ norm is over the probability space $\Omega$ of $\beta(t)$

\end{corollary}

\begin{proof}
Using Theorem $1$, we just take the $L^p$ norm of $|\beta(t)|^{\frac{-1}{2}}$. This is equivalent to

$$ \left(\frac{2}{\sqrt{2t\pi}} \int_0^{\infty} x^{\frac{-p}{2}} e^{\frac{x^2}{2t}} \, \text{d}x \right)^{\frac{1}{p}}$$

We can now substitute in $u = \frac{x}{\sqrt{t}}$, yielding

$$\left( \frac{2}{\sqrt{2\pi}}t^{\frac{-p}{4}} \int_0^{\infty} u^{\frac{-p}{2}}e^{\frac{u^2}{2}} \, \text{d}u \right) ^{\frac{1}{p}}$$

The integral converges when $p<2$, leading to growth bounded by $|t|^{\frac{-1}{4}}$
\end{proof}

To prove Theorem $1$, the concepts of resonance and energy must still be explained first.

Define Jost solutions \cite{Goldberg} $f_{\pm}(z,\cdot)$ as solutions to 

$$f_{\pm}^{''}(\lambda,x) + V(x)f_{\pm}(\lambda,x) = \lambda^2_{\pm}(\lambda,x)$$

 satisfying $|f_{\pm}(\lambda,x)-e^{\pm ix\lambda}| \rightarrow 0$ as $x \rightarrow \pm \infty$. For $V \in L_1^1(\mathbb{R})$, these Jost solutions exist for all $\lambda \in \mathbb{R}$. Define the Jost solutions of $\lambda$ as having energy $\lambda^2$

Next, define the Wronskian $W(f_1,...,f_n)$ of $n$ functions as the determinant formed by the functions and their derivatives up to order $n-1$. Denote $W(\lambda)$ as $W(f_{+}(\lambda,\cdot),f_{-}(\lambda,\cdot))$. There is a resonance at zero energy if $W(0) = 0$.

We proceed with the proof of Theorem 1 by adapting the method used by \cite{Goldberg}. We start by considering the high energy case.

\subsection{\centering High-energy Dispersive Estimates}

Let $\lambda_0 = ||V||_1^2$ and let $\chi$ be a smooth cut-off such that $\chi(\lambda)=0$ for $\lambda \leq \lambda_0$ and $\chi(\lambda)=1$ for $\lambda \geq 2\lambda_0$. 

\begin{lemma}
$$||e^{-i\beta(t)H}\chi(H)||_{1 \rightarrow \infty} \lesssim |\beta(t)|^{\frac{-1}{2}}$$
\end{lemma}

\vspace{0.5cm}

Note that the resolvent $R_V$ equals $(H-(\lambda+i\epsilon))^{-1}$. Define $R_V(\lambda \pm i0)(x) = \lim_{\epsilon \rightarrow 0^+} R_V(\lambda \pm i\epsilon)(x)$.  Thus, 
$$R_0(\lambda \pm i0)(x) = \frac{\pm i}{2\sqrt{\lambda}}e^{\pm i|x|\sqrt{\lambda}}$$

using Fourier analysis. We then apply the Born expansion, which says for any $L_1$ functions $f,g$
$$\langle R_V(\lambda \pm i0)f,g \rangle = \sum_{n=0}^{\infty} \langle R_0(\lambda \pm i0)(-VR_0(\lambda \pm i0))^nf,g\rangle$$

For $\lambda>0$, note that $R_V(\lambda-i0)g \in L^{\infty}$. Further expansion leads to

$$|\langle R_V(\lambda+i0)(VR_0(\lambda+i0))^nf,g\rangle | \leq (2\sqrt{\lambda})^{-n}||V||_1^n||f||_1||R_V(\lambda-i0)g||_{\infty}$$

We now introduce another function $\chi_L$. Let $L \geq 1$ and let $\phi$ be a smooth function with $\phi(\lambda)=1$ if $|\lambda| \leq 1$ and $\phi(\lambda)=2$ if $|\lambda| \geq 2$. Additionally, denote $E$ as the spectral measure and $E_{ac}$ as the absolutely continuous part. Note that $\chi_L(H)E(\, \text{d} \lambda) = \chi_L(H) E_{ac}(\, \text{d} \lambda)$

\begin{proposition}[Stone's Formula]

    \[
    \frac{1}{2}\bigg{(}\mu^{A}([a,b]) + \mu^A((a,b))\bigg{)} = \frac{1}{2\pi i} \int_a^b R(z + i0) -R(z-i0) \, \text{d}z \tag{*}
    \]

\end{proposition}

\begin{proof}

Using functional calculus, write the right side as

$$\lim_{\epsilon \rightarrow 0^+} \frac{1}{2\pi i}\int_a^b \left( \frac{1}{u-(z+i\epsilon)} - \frac{1}{u-(z-i\epsilon)} \right) \, \text{d}z$$

This expression equals $0$ if $u \notin [a,b]$ and equals $1$ if $u \in (a,b)$. Additionally, if $u = a$ or $u = b$, the expression equates to $\frac{1}{2}$. Thus, after applying functional calculus, this is equivalent to 

$$\frac{1}{2}\left(\mu^A([a,b]) + \mu^A((a,b))\right)$$

\end{proof}

From Stone's Formula, we know

$$\langle E_{ac}(\, \text{d}\lambda)f,g\rangle = \langle \frac{1}{2i\pi} [R_v(\lambda+i0)-R_V(\lambda-i0)]f,g \rangle \, \text{d}\lambda$$

We now make the change $\lambda \rightarrow \lambda^2$. Here, note the equality $R_0(\lambda^2-i0) = R_0((-\lambda)^2+i0)$, which allows us to extend the range of integration. 

$$|\langle e^{-i\beta(t)H}\chi_L(H)f,g\rangle| = |(2i\pi)^{-1}  \int_{0}^{\infty} \sum_{n=0}^{\infty}  e^{-i\beta(t)\lambda}\chi_L(\lambda)\lambda\langle (R_0(\lambda+i0)(VR_0(\lambda+i0))^n - R_0(\lambda-i0)(VR_0(\lambda-i0))^n)f,g\rangle \, \text{d}\lambda |$$

$$= |(2i\pi)^{-1} \sum_{n=0}^{\infty} \int_{0}^{\infty}   e^{-i\beta(t)\lambda}\chi_L(\lambda)\lambda\langle (R_0(\lambda+i0)(VR_0(\lambda+i0))^n - R_0(\lambda-i0)(VR_0(\lambda-i0))^n)f,g\rangle \, \text{d}\lambda |$$

$$= |(2i\pi)^{-1} \sum_{n=0}^{\infty} \int_{0}^{\infty}   e^{-i\beta(t)\lambda^2}\chi_L(\lambda^2)\lambda\langle (R_0(\lambda^2+i0)(VR_0(\lambda^2+i0))^n - R_0(\lambda^2-i0)(VR_0(\lambda^2-i0))^n)f,g\rangle \, \text{d}\lambda |$$

$$= |(2i\pi)^{-1} \sum_{n=0}^{\infty} \int_{-\infty}^{\infty} e^{-i\beta(t)\lambda^2}\chi_L(\lambda^2)\lambda\langle R_0(\lambda^2+i0)(VR_0(\lambda^2+i0))^nf,g\rangle \, \text{d}\lambda |$$

To work with this, first make the substitution

$$R_0(\lambda^2+i0)(VR_0(\lambda^2+i0))^n(x,y) = \frac{1}{(2\lambda)^{n+1}} \int_{\mathbb{R}^n} \prod_{j=1}^{n} V(x_j) e^{-i\lambda(|x-x_1|+|y-x_n|+\sum_{k=2}^{n} |x_k-x_{k-1}|)} \, \text{d} x_1 ... \, \text{d} x_n$$

We first integrate by $\lambda$, which yields

$$|\langle e^{-i\beta(t)H}\chi_L(H)f,g\rangle| \lesssim \sum_{n=0}^{\infty} (2\sqrt{\lambda_0})^{-n} \sup_{a \in \mathbb{R}} | \int_{-\infty}^{\infty} e^{-i(\beta(t)\lambda^2+a\lambda)}\chi_L(\lambda^2)\lambda^{-n}\lambda_0^{n/2} \, \text{d}\lambda||||V||_1^n||f||_1||g||_1$$

Note that the quantity inside the absolute value is equivalent to the solution of equation \ref{eq1} at time $t$ and position $a$, and initial data $[\chi_L(\lambda^2)\lambda^{-n}\lambda_0^{n/2}]^{\vee}$

\vspace{0.5cm}

For equation \ref{eq1}, the solution given initial condition $f_0$ is given by $\frac{1}{\sqrt{-4i\beta(t)}} \int_{\mathbb{R}} e^{\frac{-i|x-y|^2}{\beta(t)}} f_0 \, \text{d}y$. The solution grows at rate $|\beta(t)|^{\frac{-1}{2}}$ if we disregard the initial condition component. Here, the initial condition is $f_0 = \chi_L(\lambda^2)\lambda^{-n}\lambda_0^{n/2}$. Thus, we must show that $\sup_{L \geq 1,n \geq 0} ||[\chi_L(\lambda^2)\lambda^{-n}\lambda_0^{n/2}]^{\vee}||_1 < \infty$

where $\vee$ represents the inverse Fourier Transform.

If $n=0$, this yields

$$||[\chi_L(\lambda^2)]^{\vee}||_1 \leq ||L\widehat{\phi(\lambda^2)}(L\zeta)||_1(1+||(1-\chi(\lambda^2))^{\vee}||_1)<\infty$$

For $n \geq 2$,

$$||[\chi_L(\lambda^2)\lambda^{-n}]^{\vee}(\zeta)||_{\infty} \leq ||\chi_L(\lambda^2)\lambda^{-n}||_1\leq C(\lambda_0)\lambda_0^{-n/2}$$

where $C(\lambda_0)$ only depends on $\lambda_0$

For the $n=1$ case,

$$||[\chi_L(\lambda^2)\lambda^{-1}]^{\vee}||_{\infty} \leq ||[\chi_L(\lambda^2)]^{\vee}||_1||[\lambda^{-1}]^{\vee}||_{\infty} < \infty$$

when $L \geq 1$

Thus, 

$$\sup_{L \geq 1,n \geq 0} ||[\chi_L(\lambda^2)\lambda^{-n}\lambda_0^{n/2}]^{\vee}||_1 < \infty$$

concluding the high-energy part and Lemma $3.1$.

\subsection{\centering Low-energy Dispersive Estimates}

For low-energy portion, note that for $x<y$,

$$R_V(\lambda^2 \pm i0)(x,y) = \frac{f_{+}(\pm \lambda,y)f_{-}(\pm \lambda,y)}{W(\pm \lambda)}$$

where $f$ are the Jost-solutions \cite{Goldberg}. 

Using Stone's formula,

$$2i\pi \int_{0}^{\infty} e^{-i\beta(t)\lambda}\chi(\lambda)E_{ac}(\, \text{d} \lambda)(x,y) $$

$$= \int_{0}^{\infty} e^{-i\beta(t)\lambda^2}\lambda \chi(\lambda^2) [\frac{f_{+}(\lambda,y)f_{-}(\lambda,x)}{W(\lambda)}-\frac{f_{+}(-\lambda,y)f_{-}(-\lambda,x)}{W(-\lambda)}] \, \text{d} \lambda$$

$$= \int_{-\infty}^{\infty} e^{-i\beta(t)\lambda^2}\lambda \chi(\lambda^2) \frac{f_{+}(\lambda,y)f_{-}(\lambda,x)}{W(\lambda)} \, \text{d} \lambda$$

Write $f_{\pm}(\lambda,x)$ as $e^{\pm ix\lambda}m_{\pm}(\lambda,x)$. Using \cite{Deift}, the function $m_{\pm}(z,x)-1$ belongs to the Hardy Space on the upper half plane, allowing for various estimates on its Fourier Transform in the first variable. Let $m_{\pm}(\hat{\zeta},x)$ denote this Fourier Transform in the first variable. 

We take the following Lemma from \cite{Goldberg}.

\begin{lemma}
Let $V \in L_j^1(\mathbb{R})$, $j=1,2$ and $\tilde{\chi}$ be a smooth, compactly supported cut-off function which is $1$ on the support of $\chi$. Then $\tilde{\chi}(\lambda)W(\lambda)$ and $W[f_{+}(\lambda,\cdot),f_{-}(-\lambda,\cdot)]$ both have Fourier transform in $L_{j-1}^1(\mathbb{R})$
\end{lemma}

But we can now apply Lemma $3.2$ to equation \ref{eq3} to yield another lemma for Strichartz estimates. 
\begin{lemma}
$$\sup_{x<y}|\int_{-\infty}^{\infty} e^{-i\beta(t)\lambda^2}\frac{\lambda \chi(\lambda)}{W(\lambda)}f_{+}(\lambda,y)f_{-}(\lambda,x) \, \text{d}\lambda| \lesssim |\beta(t)|^{\frac{-1}{2}}$$
\end{lemma} 

\begin{proof}
We must consider the three cases $x < 0 < y$, $0 < x < y$, and $x < y < 0$.

If $x < 0 < y$:

$$\sup | \int_{-\infty}^{\infty} e^{-i\beta(t)\lambda^2} \frac{\lambda \chi(\lambda)}{W(\lambda)}f_+(\lambda,y)f_-(\lambda,x) \, \text{d}\lambda|$$

$$= \sup | \int_{-\infty}^{\infty} e^{-i\beta(t)\lambda^2} e^{-i\lambda(x-y)} \frac{\lambda \chi(\lambda)}{\tilde{\chi}(\lambda)W(\lambda)}m_+(\lambda,y)m_-(\lambda,x) \, \text{d}\lambda|$$

$$\lesssim |\beta(t)|^{\frac{-1}{2}} $$

The last inequality follows in the same way as the high-energy part we apply the dispersive bound from the Schrodinger equation and

$$\sup ||[\frac{\lambda \chi(\lambda)}{\tilde{\chi}(\lambda)W(\lambda)}m_+(\lambda,y)m_-(\lambda,x)]^{\vee}|| < \infty$$

This comes from Lemma $4$ where $[\tilde{\chi}W]^{\vee} \in L^1(\mathbb{R})$ allows us to apply Wiener's lemma so that $\frac{\lambda \chi(\lambda)}{\tilde{\chi}(\lambda)W(\lambda)}$ is the Fourier Transform of an $L^1$ function.

The other two cases are similar due to symmetry, so we consider the case $0 \leq x < y$. 

Write $f_-(\lambda,x) = \alpha(\lambda)f_+(\lambda,x)+\beta(\lambda)f_+(-\lambda,x)$ with $\beta(\lambda)=\frac{W(\lambda)}{-2i\lambda}$ and $\alpha(\lambda) = \frac{-1}{2i\lambda}W[f_-(\lambda,\cdot),f_+(-\lambda,\cdot)]$

Thus, 

$$\sup | \int_{-\infty}^{\infty} e^{-i\beta(t)\lambda^2} \frac{\lambda \chi(\lambda)}{\tilde{\chi}(\lambda)W(\lambda)}f_+(\lambda,y)f_-(\lambda,x) \, \text{d}\lambda|$$

$$\lesssim \sup | \int_{-\infty}^{\infty} e^{-i\beta(t)\lambda^2} e^{-i\lambda(x+y)} \frac{\lambda \alpha(\lambda)}{\tilde{\chi}(\lambda)W(\lambda)}\chi(\lambda)m_+(\lambda,y)m_+(\lambda,x) \, \text{d}\lambda| + | \int_{-\infty}^{\infty} e^{-i\beta(t)\lambda^2} e^{-i\lambda(y-x)} \chi(\lambda)m_+(\lambda,y)m_+(-\lambda,x) \, \text{d}\lambda|$$

$$ \lesssim |\beta(t)|^{\frac{-1}{2}}$$
\end{proof}

\section{\centering Strichartz Estimates}

In the previous section, we developed various dispersive estimates. These can be used to create Strichartz estimates for equation $\ref{eq3}$.

\begin{lemma}

Let $u_s$ be some initial condition at time $s$. Let $S(t,s)$ be the solution operator for equation \ref{eq3} given initial condition at time $s$ with $s \leq t$, which we know is equal to $e^{\i(t-s)\beta(t-s)}$ Then for any $p \geq 2$, $S(t,s)$ maps $L_x^{p'}$ into $L_x^p$ with constant $C_p$ such that

$$|S(t,s)P_{ac}(H)u_s|_{L_x^p} \leq \frac{C_p}{|\beta(t)-\beta(s)|^{\frac{1}{2}-\frac{1}{p}}} |u_s|_{L_x^{p'}}$$
\end{lemma}

\begin{proof}
    Theorem $1$ covers the case $p = \infty$. Additionally, $S(t,s)$ is an isometry on $L_x^2$. Thus, the result then follows from Riesz-Thorin Interpolation.
\end{proof}

From \cite{DeBouard}, we take the following lemma without proof.

\begin{lemma}
Let $\alpha \in [0,1)$, then there exists a constant $c_{\alpha}$ depending only on $\alpha$ such that for any $T \geq 0$ and $f \in L_{\mathcal{P}}^2(\Omega; L^2(0,T))$

$$\mathbb{E} \left( \int_0^T \left( \int_0^t \frac{1}{|\beta(t)-\beta(s)|^{\alpha}} |f(s)| \, \text{d}s \right)^2 \, \text{d}t \right) \leq c_{\alpha} T^{2-\alpha} \mathbb{E} \left( \int_0^T |f(s)|^2 \, \text{d}s \right)$$
\end{lemma}

\begin{lemma}
Let $\alpha = 0$, $r \leq \infty$ or $\alpha \in (0,1)$, $2 \leq r < \frac{2}{\alpha}$ and let $\rho$ be such that $r' \leq \rho \leq r$. Then there exists $C$ such that for any $T \geq 0$ and $f \in L^p_{\rho}(\Omega; L^{r'}(0,T))$,

$$\big| \int_0^t |\beta(t) - \beta(s)| ^{-\alpha}|f(s)| \, \text{d}s \big|_{L^p_{\omega}(L^r(0,T))} \leq CT^{\frac{2}{r}-\frac{\alpha}{2}}|f|_{L^\rho(\Omega;L^{r'}(0,T))}$$
\end{lemma}

\begin{proof}
First, plug in $\alpha = 0$ and see the equation obviously holds when $r = \infty$. Additionally, if $\alpha < 1$ and $r = \rho = 2$, this is Lemma $4.1$. Using these two cases, we can use interpolation to generate the full result.

\end{proof}

\begin{theorem***}
    Let $2 \leq r < \infty$ and $2 \leq p \leq \infty$ such that $\frac{2}{r} > \frac{1}{2} - \frac{1}{p}$ or $r=\infty$ and $p=2$. Additionally, set $\rho$ such that $r' \leq \rho \leq r$, where $r'$ is the Holder conjugate of $r$. Then, there exists a constant $c_{\rho,r,p}$ such that for any $s \in \mathbb{R}$, $T \leq 0$, and $f \in L^{\rho}_{\mathcal{P}}(\Omega; L^{r'}(s,s+T;L_x^{p'}))$

    $$|\int_s^{\cdot} S(\cdot,\sigma)P_{ac}(H)f(\sigma) \, \text{d} \sigma|_{L^{\rho}(\Omega; L^{r}(s,s+T;L_x^{p}))} \geq c_{\rho,r,p}T^{\mu}|f|_{L^{\rho}(\Omega; L^{r'}(s,s+T;L_x^{p'}))} $$
with $\mu = \frac{2}{r} + \frac{1}{2p} - \frac{1}{4}$
\end{theorem***}

\begin{proof}

    Let $(r,p)$ be an admissible pair if $r = \infty$ and $p = 2$ or if $2 \leq r < \infty, 2 \leq p \leq \infty$, and $\frac{2}{r} > \frac{1}{2} - \frac{1}{p}$

    Choose some admissible pair $(r,p)$. Choose $\rho$ such that $r' \leq \rho \leq r$ and let $f \in L_{\mathcal{P}}^{\rho}(\Omega; L^{r'}(s,s+T;L_x^{p'}))$. Using Lemma $6$, note

    $$\big| \int_s^t S(t,\sigma) P_{ac}(H) f(\sigma) \, \text{d} \sigma \big|_{L_x^{p}}$$

    $$\leq \int_s^t |S(t,\sigma) P_{ac}(H)f(\sigma)|_{L_x^p} \, \text{d} \sigma$$

    $$\leq c \int_s^t \frac{1}{|\beta(t)-\beta(\sigma)|^{\frac{1}{2}-\frac{1}{p}}}|f(\sigma)|_{L_x^{p'}} \, \text{d}\sigma$$

    Using Lemma $4.3$ with $\alpha = \frac{1}{2} - \frac{1}{p} \in [0,1)$, we see

    $$\mathbb{E} \left( \big| \int_s^{\cdot} S(\cdot,\sigma) P_{ac}(H) f(\sigma) \, \text{d}\sigma \big|_{L^r(s,s+T;L_x^p)}^{\rho} \right) \leq cT^{\rho(\frac{2}{r}+\frac{1}{2p}-\frac{1}{4})}|f|^{\rho}_{L^{\rho}(\Omega;L^{r'}(s,s+t;L_x^{p'}))}$$

    leading to the theorem.
\end{proof}
    \begin{theorem***}
    Let $2 \leq r < \infty$ and $2 \leq p \leq \infty$ be such that $\frac{2}{r} > \frac{1}{2} - \frac{1}{p}$ or $r = \infty$ and $p = 2$. Then there exists a constant $c_{r,p} > 0$ such that for $s \in \mathbb{R}, T \geq 0$, $u_s \in L^r(\Omega;L_x^2)$,

    $$S(\cdot,s)P_{ac}(H)u_s \in L^r_{\mathcal{P}}(\Omega;L^r(s,s+T;L_x^p))$$

    and 

    $$|S(\cdot,s)P_{ac}(H)u_s|_{L^r(\Omega;L^r(s,s+T;L^p_x))} \leq cT^{\frac{\mu}{2}}|u_s|_{L^r_{\omega}(L_x^2)}$$

    where $\mu = \frac{2}{r} - \frac{1}{2}(\frac{1}{2} - \frac{1}{p})$

    \text{Specifically},

    $$|e^{-i\beta(t)H}P_{ac}(H)u_0|_{L^r(\Omega;L^r(0,T;L^p_x))} \leq cT^{\frac{\mu}{2}}|u_0|_{L^r_{\omega}(L_x^2)}$$

    \end{theorem***}

\begin{proof}
Let the adjoint of $S(t,s)$ be $S(s,t)$ with the adjoint taken with respect to the $L_x^2$ inner product. Thus, when $f \in L_{\mathcal{P}}^2(\Omega \times [s,s+T] \times \mathbb{R})$ and $u_s \in L^r(\Omega; L_x^2)$, we can use the adjoint to see

$$\int_s^{s+T} (S(t,s)P_{ac}(H)u_s,f(t)) \, \text{d}t$$

$$ = \int_s^{s+T} (u_s,S(s,t)P_{ac}(H)f(t)) \, \text{d}t$$

$$ \leq |u_s|_{L_x^2} \big| \int_s^{s+T} S(s,t)P_{ac}(H)f(t) \, \text{d}t \big|_{L_x^2}$$

We also see

$$\big| \int_s^{s+T} S(s,t)P_{ac}(H)f(t) \, \text{d}t \big|^2_{L_x^2}$$

$$ = \int_s^{s+T} \int_s^{s+T} (f(t),S(t,\sigma)P_{ac}(H)f(\sigma)) \, \text{d}t \, \text{d} \sigma$$

$$ = \int \int_{s \leq \sigma \leq t \leq s+T} (f(t),S(t,\sigma)P_{ac}(H)f(\sigma)) \, \text{d}t \, \text{d} \sigma + \int \int_{s \leq t \leq \sigma \leq s+T} (f(t),S(t,\sigma)P_{ac}(H)f(\sigma)) \, \text{d}t \, \text{d} \sigma$$

$$ = 2\int \int_{s \leq \sigma \leq t \leq s+T} (f(t),S(t,\sigma)P_{ac}(H)f(\sigma)) \, \text{d}t \, \text{d} \sigma $$

$$ = 2 \int_s^{s+T} \left( f(t), \int_s^t S(t,\sigma)P_{ac}(H)f(\sigma) \, \text{d} \sigma \right) \, \text{d}t $$

$$\leq 2|f|_{L^{r'}(s,s+T;L_x^{p'})} \big| \int_s^{\cdot} S(\cdot,\sigma)P_{ac}(H)f(\sigma) \, \text{d} \sigma \big|_{L^r(s,s+T;L_x^p)}$$

We can then use theorem $2$ to see

$$\mathbb{E} \int_s^{s+T} (S(t,s)u_s,f(t)) \, \text{d}t$$

$$\leq \sqrt{2} \mathbb{E} \left( |u_s|_{L^2_x}|f|^{\frac{1}{2}}_{L^{r'}(s,s+T;L_x^{p'})} \big| \int_s^{\cdot} S(\cdot,\sigma)P_{ac}(H)f(\sigma)\, \text{d} \sigma\big|^{\frac{1}{2}}_{L^r(s,s+T;L_x^p)} \right)$$

$$\leq \sqrt{2}  |u_s|_{L^2_x}|f|^{\frac{1}{2}}_{L_{\omega}^{r'}(L^{r'}(s,s+T;L_x^{p'}))} \big| \int_s^{\cdot} S(\cdot,\sigma)P_{ac}(H)f(\sigma)\, \text{d} \sigma\big|^{\frac{1}{2}}_{L_{\omega}^{r'}(L^r(s,s+T;L_x^p))} $$

$$\leq cT^{\frac{\mu}{2}}|u_s|_{L_{\omega}^r(|_x^2)}|f|_{L_{\omega}^{r'}(L^{r'}(s,s+T;L_x^{p'}))}$$

with $\mu = \frac{2}{r} - \frac{1}{2}(\frac{1}{2} - \frac{1}{p})$
\end{proof}

\vspace{3cm}

\printbibliography

\end{document}